\documentclass{amsart}%
\usepackage{amsfonts}%
\usepackage{amsmath}%
\usepackage{amssymb}%
\usepackage{graphicx}
\newtheorem{theorem}{Theorem}
\theoremstyle{plain}
\newtheorem{acknowledgement}{Acknowledgement}

\newtheorem{definition}{Definition}

\newtheorem{lemma}{Lemma}

\newtheorem{proposition}{Proposition}
\newtheorem{remark}{Remark}

\numberwithin{equation}{section}
\begin{document}
\title[A characterization of the hyperbolic space]{Harmonic functions, entropy, and a characterization of the hyperbolic space}
\author{Xiaodong Wang}
\address{Department of Mathematics\\
Michigan State University\\
East Lansing, MI\ 48824}
\email{xwang@math.msu.edu}
\thanks{Partially supported by NSF Grant 0505645.}

\begin{abstract}
Let $(M^{n},g)$ be a compact Riemannian manifold with $Ric\geq-\left(
n-1\right)  $. It is well known that the bottom of spectrum $\lambda_{0}$ of
its unverversal covering satisfies $\lambda_{0}\leq\left(  n-1\right)  ^{2}/4
$. We prove that equality holds iff $M$ is hyperbolic. This follows from a
sharp estimate for the Kaimanovich entropy.

\end{abstract}
\date{}
\maketitle

\section{Introduction}

Complete Riemannian manifolds with nonnegative Ricci curvature have been
intensively studied by many people and there are various methods and many
beautiful results (see e.g., the book \cite{P}). One of the most important
theorems on such manifolds is the following Cheeger-Gromoll splitting theorem:

\begin{theorem}
(Cheeger-Gromoll) If $\left(  N,g\right)  $ contains a line and has $Ric\geq
0$, then $\left(  N,g\right)  $ is isometric to a product $\left(
\mathbb{R}
\times\Sigma,dt^{2}+h\right)  $.
\end{theorem}

This theorem has the following important corollaries on the structure of
manifolds with nonnegative Ricci curvature:

\begin{itemize}
\item A complete Riemannian $\left(  N,g\right)  $ with $Ric\geq0$ either has
only one end or is isometric to a product $\left(
\mathbb{R}
\times\Sigma,dt^{2}+g_{\Sigma}\right)  $, with $\Sigma$ compact.

\item If $\left(  M^{n},g\right)  $ is compact with $Ric\geq0$ then its
universal covering $\widetilde{M}$ splits isometrically as a product $%
\mathbb{R}
^{k}\times\Sigma^{n-k}$, where $\Sigma$ is a simply connected compact manifold
with $Ric\geq0$. If furthermore $\widetilde{M}$ has Euclidean volume growth,
then $\widetilde{M}$ is isometric to $%
\mathbb{R}
^{n}$.
\end{itemize}

Riemannian manifolds with a negative lower bound for Ricci curvature are
considerably more complicated and less understood. It is naive to expect such
splitting results in general. Nevertheless there have been very interesting
results due to Li and J. Wang recently. It has been discovered that the bottom
of the $L^{2}$ spectrum plays an important role (see also the earlier work
\cite{W} in the conformally compact case). Let us assume that $\left(
N^{n},g\right)  $ is a complete Riemannian manifold with $Ric\geq-\left(
n-1\right)  $. The bottom of the $L^{2}$ spectrum of the Laplace operator on
functions is denoted by $\lambda_{0}\left(  N\right)  $ and can be
characterized as%
\[
\lambda_{0}\left(  N\right)  =\inf\frac{\int_{N}\left\vert \nabla u\right\vert
^{2}}{\int_{N}u^{2}},
\]
where the infimum is taken over all smooth functions with compact support. It
is well known $\lambda_{0}\left(  N\right)  \leq\left(  n-1\right)  ^{2}/4$.
Li-Wang proved the following theorems for manifolds with positive $\lambda
_{0}$.

\begin{theorem}
(Li-Wang) Let $\left(  N^{n},g\right)  $ be a complete Riemannian manifold
with $Ric\geq-\left(  n-1\right)  $ and $\lambda_{0}\left(  N\right)  \geq
n-2$. Then either

\begin{enumerate}
\item $N$ has only one end with infinite volume; or

\item $N$ $=%
\mathbb{R}
\times\Sigma$ with warped product metric $g=dt^{2}+\cosh^{2}tg_{\Sigma}$,
where $\left(  \Sigma,g_{\Sigma}\right)  $ is compact with $Ric$ $\geq-\left(
n-2\right)  $.
\end{enumerate}
\end{theorem}

When $\lambda_{0}\left(  N\right)  =\left(  n-1\right)  ^{2}/4$, they can also
handle ends of finite volume.

\begin{theorem}
\label{LW2}(Li-Wang) Let $\left(  N^{n},g\right)  $ be a complete Riemannian
manifold with $Ric\geq-\left(  n-1\right)  $ and $\lambda_{0}\left(  N\right)
=\left(  n-1\right)  ^{2}/4$. Then either

\begin{enumerate}
\item $N$ has only one end; or

\item $N$ $=%
\mathbb{R}
\times\Sigma$ with warped product metric $g=dt^{2}+e^{2t}g_{\Sigma}$, where
$\left(  \Sigma,h\right)  $ is compact manifold with nonnegative Ricci curvature;

\item $n=3$ and $N=%
\mathbb{R}
\times\Sigma$ with warped product metric $g=dt^{2}+\cosh^{2}tg_{\Sigma}$,
where $\left(  \Sigma,h\right)  $ is compact surface with Gaussian curvature
$\geq-1$.
\end{enumerate}
\end{theorem}

The basic point is that $\lambda_{0}\left(  N\right)  $ is sensitive to the
connectedness at infinity. In both cases if there are two ends they are able
to prove that the manifold splits as a warped product. Since the splitting is
only obtained under this restrictive situation, their theorems are not as
powerful as the Cheeger-Gromoll splitting theorem is for manifolds with
nonnegative Ricci curvature. An intriguing question is if there is a more
general mechanism under which a complete $\left(  N^{n},g\right)  $ with
$Ric\geq-\left(  n-1\right)  $ and $\lambda_{0}\left(  N\right)  =\left(
n-1\right)  ^{2}/4$ splits as a warped product.

We will consider a special situation inspired by an aforementioned corollary
of the Cheeger-Gromoll theorem. Suppose $\widetilde{M}$ is the universal
covering of a compact Riemannian manifold $\left(  M^{n},g\right)  $ with
$Ric\geq-\left(  n-1\right)  $. What happens if $\lambda_{0}\left(
\widetilde{M}\right)  =\left(  n-1\right)  ^{2}/4$, the largest possible
value? From Theorem \ref{LW2} it seems the only information we can draw is
that $\widetilde{M}$ has only one end. On the other hand it is reasonable to
expect that $\widetilde{M}$ is isometric to $\mathbb{H}^{n}$, i.e. $\left(
M^{n},g\right)  $ is a hyperbolic manifold. In fact this has been conjectured
by Jiaping Wang.

The main result of this paper is an affirmative answer to this conjecture,
i.e. we prove

\begin{theorem}
[Main Theorem]\label{main}Let $\left(  M^{n},g\right)  $ be a compact
Riemannian manifold with $Ric\geq-\left(  n-1\right)  $ and $\pi:\widetilde
{M}\rightarrow M$ its universal covering. If $\lambda_{0}\left(  \widetilde
{M}\right)  =\left(  n-1\right)  ^{2}/4$, then $\widetilde{M}$ is isometric to
the hyperbolic space $\mathbb{H}^{n}$.
\end{theorem}

It is worth pointing out that complete Riemannian manifolds $N^{n}$ with
$Ric\geq-\left(  n-1\right)  $ and $\lambda_{0}=\left(  n-1\right)  ^{2}/4$
are abundant. There are many conformally compact examples by a theorem of Lee
\cite{Lee}. Therefore the more subtle assumption that $\widetilde{M}$ covers a
compact manifold is essential.

If we further assume that $g$ is negatively curved, then the result follows
from the following well known theorem. Recall the volume entropy $h$ is
defined to be%
\[
h=\lim_{r\rightarrow+\infty}\frac{\log V(p,r)}{r},
\]
where $V\left(  p,r\right)  $ is the volume of the geodesic ball with center
$p $ and radius $r$ in $\widetilde{M}$.

\begin{theorem}
Let $\left(  M^{n},g\right)  $ be a compact Riemannian manifold with negative
curvature and $\pi:\widetilde{M}\rightarrow M$ its universal covering. If
$\lambda_{0}\left(  \widetilde{M}\right)  =h^{2}/4$, then $M$ is locally symmetric.
\end{theorem}

This is a deep theorem whose proof is difficult and involved : first it is
proved by Ledrappier \cite{L1} that $\widetilde{M}$ is asymptotically harmonic
in the sense that all the level sets of the Buseman functions have constant
mean curvature; then by a theorem of Foulon and Labourie \cite{FL} the
geodesic flow of $\left(  M,g\right)  $ is $C^{\infty}$-conjugate to that of a
locally symmetric space of rank one; finally by the work of
Besson-Courtois-Gallot \cite{BCG} one concludes that $\left(  M,g\right)  $ is
locally symmetric.

In contrast, we make no additional assumption on the sectional curvature and
prove the Main Theorem in a direct way. A key ingredient in our proof is the
entropy introduced by Kaimanovich \cite{K}, which we learned from \cite{L1}.
The paper is organized as follows. In Section 2, we discuss positive harmonic
functions and the Martin boundary. In Section 3, we present the Kaimanovich
entropy. We establish a sharp inequality which characterizes hyperbolic
manifolds. The main theorem is then proved in Section 4.

\begin{acknowledgement}
I wish to thank Jiaping Wang for bringing to my attention this problem and for
some stimulating discussions on his joint work with Li. I am very grateful to
Professor F. Ledrappier for explaining Kaimanovich's work to me and to
Professor J. Cao for his interest and encouragement.
\end{acknowledgement}

\section{Harmonic functions and the Martin boundary}

In this section \ we collect some fundamental facts on positive harmonic
functions. There are two aspects: potential theory and geometric analysis. On
potential theory (more specifically, the theory of Martin boundary) our
primary reference is Ancona \cite{A} (\cite{AG} and \cite{H} are also very
useful) . Let $\widetilde{M}^{n}$ be a complete Riemannian manifold with a
base point $o$. We assume that $\widetilde{M}$ is non-parabolic, that is, it
has a positive Green's function. It is well known that $\widetilde{M}$ is
non-parabolic if $\lambda_{0}\left(  \widetilde{M}\right)  >0$ (see, e.g.
\cite{SY}). The vector space $\mathcal{H}\left(  \widetilde{M}\right)  $ of
harmonic functions with seminorms
\[
\left\Vert u\right\Vert _{K}=\sup_{K}\left\vert u\left(  x\right)  \right\vert
,K\subset\widetilde{M}\text{ compact}%
\]
is a Frechet space. Let $\mathcal{K}_{o}=\{u\in\mathcal{H}\left(  M\right)
:u\left(  o\right)  =1,u>0\}$. This is a convex and compact set.

\begin{definition}
A harmonic function $h>0$ on $\widetilde{M}$ is called minimal if any
nonnegative harmonic function $\leq h$ is proportional to $h$.
\end{definition}

\begin{remark}
If $h\left(  o\right)  =1$, then $h$ is minimal iff $h$ is an extremal point
of $\mathcal{K}_{o}$.
\end{remark}

\begin{definition}
The minimal Martin boundary of $M$ is
\[
\partial^{\ast}\widetilde{M}=\{h\in\mathcal{K}_{o}:h\text{ is minimal}\}.
\]

\end{definition}

According to a theorem of Choquet (\cite{A}), for any positive harmonic
function $h$ there is a unique Borel measure $\mu^{h}$ on $\partial^{\ast
}\widetilde{M}$ such that%
\[
h\left(  x\right)  =\int_{\partial^{\ast}\widetilde{M}}\xi\left(  x\right)
d\mu^{h}\left(  \xi\right)  .
\]
Let $\nu$ be the measure corresponding to the harmonic function $1$. Thus%
\begin{equation}
1=\int_{\partial^{\ast}\widetilde{M}}\xi\left(  x\right)  d\nu\left(
\xi\right)  .\label{defnu}%
\end{equation}
For $f\in L^{\infty}\left(  \partial^{\ast}\widetilde{M}\right)  $ we get a
bounded harmonic function%
\[
H_{f}\left(  x\right)  =\int_{\partial^{\ast}\widetilde{M}}f\left(
\xi\right)  \xi\left(  x\right)  d\nu\left(  \xi\right)  .
\]

When there is a lower Ricci bound, Yau's gradient estimate for positive
harmonic functions (see e.g. \cite{SY}) is a very powerful tool. The following
sharp version is due to Li-Wang \cite{LW2}.

\begin{lemma}
\label{ge}Let $\left(  N^{n},g\right)  $ be a complete Riemannian manifold
with $Ric\geq-\left(  n-1\right)  $. For a positive harmonic function $f$ on
$N$ we have%
\[
\left\vert \nabla\log f\right\vert \leq n-1
\]
on $N$.
\end{lemma}

Let $\phi=\log f$. Then $\Delta\phi=-\left\vert \nabla\phi\right\vert ^{2}%
$\thinspace. Using the Bochner formula we have%
\begin{align}
\frac{1}{2}\Delta\left\vert \nabla\phi\right\vert ^{2}  & =\left\vert
D^{2}\phi\right\vert ^{2}+\left\langle \nabla\phi,\nabla\Delta\phi
\right\rangle +Ric\left(  \nabla\phi,\nabla\phi\right) \label{Bochner}\\
& \geq\left\vert D^{2}\phi\right\vert ^{2}-\left\langle \nabla\phi
,\nabla\left\vert \nabla\phi\right\vert ^{2}\right\rangle -\left(  n-1\right)
\left\vert \nabla\phi\right\vert ^{2}.\nonumber
\end{align}
Using the equation of $\phi$ and some algebra one can derive the following
inequality at any point where $\nabla\phi\neq0$
\begin{equation}
\left\vert D^{2}\phi\right\vert ^{2}\geq\frac{n}{4\left(  n-1\right)  }%
\frac{\left\vert \nabla\left\vert \nabla\phi\right\vert ^{2}\right\vert ^{2}%
}{\left\vert \nabla\phi\right\vert ^{2}}+\frac{\left\vert \nabla
\phi\right\vert ^{4}+\left\langle \nabla\phi,\nabla\left\vert \nabla
\phi\right\vert ^{2}\right\rangle }{n-1}.\label{kato}%
\end{equation}
Moreover equality holds iff%
\begin{equation}
D^{2}\phi=-\frac{\left\vert \nabla\phi\right\vert ^{2}}{n-1}\left[  g-\frac
{1}{\left\vert \nabla\phi\right\vert ^{2}}d\phi\otimes d\phi\right]
.\label{ke}%
\end{equation}
Combining (\ref{kato}) with (\ref{Bochner}) yields%
\begin{align}
\frac{1}{2}\Delta\left\vert \nabla\phi\right\vert ^{2}  & \geq\frac
{n}{4\left(  n-1\right)  }\frac{\left\vert \nabla\left\vert \nabla
\phi\right\vert ^{2}\right\vert ^{2}}{\left\vert \nabla\phi\right\vert ^{2}%
}-\frac{n-2}{n-1}\left\langle \nabla\phi,\nabla\left\vert \nabla
\phi\right\vert ^{2}\right\rangle \label{fe}\\
& -\left(  n-1\right)  \left\vert \nabla\phi\right\vert ^{2}+\frac{\left\vert
\nabla\phi\right\vert ^{4}}{n-1}.\nonumber
\end{align}
The remaining part of the proof is to construct an appropriate cut-off
function and apply the maximum principle to show $\left\vert \nabla
\phi\right\vert \leq n-1$.

If it happens that $\left\vert \nabla\phi\right\vert \equiv n-1$, then
(\ref{fe}) is an equality. Therefore by (\ref{ke}) we have%
\[
D^{2}\phi=-\left(  n-1\right)  \left[  g-\frac{1}{\left(  n-1\right)  ^{2}%
}d\phi\otimes d\phi\right]  .
\]
One can then prove that $N$ splits as a warped product. More precisely we have
the following lemma from \cite{LW2}.

\begin{lemma}
Let $\left(  N^{n},g\right)  $ be a complete Riemannian manifold with
$Ric\geq-\left(  n-1\right)  $. If there exists a positive harmonic function
$f$ on $N$ such that $\left\vert \nabla\log f\right\vert \equiv n-1$on $N$,
then $\left(  N^{n},g\right)  $ is isometric to $\left(
\mathbb{R}
\times\Sigma^{n-1},dt^{2}+e^{2t}g_{\Sigma}\right)  $, where $\left(
\Sigma^{n-1},g_{\Sigma}\right)  $ is complete and has nonnegative Ricci
curvature. Moreover $\log f=-\left(  n-1\right)  t$.
\end{lemma}

\section{The Kaimanovich entropy}

In \cite{K} Kaimanovich studied Brownian motion on a regular covering
$\widetilde{M}$ of a compact Riemannian manifold $M$. He introduced a
remarkable entropy which will play an important role in the proof of our main
theorem. We will present his theory using the minimal Martin boundary instead
of the stationary boundary for Brownian motion. The equivalence of the two
approaches can be seen from \cite{A} (section 3 in particular). For more
detailed discussions related to the Kaimanovich entropy we refer to several
papers by Ledrappier \cite{L1,L2, L3}.

From then on we assume that $\widetilde{M}$ is the universal covering of a
compact manifold $M$. We identify $\pi_{1}\left(  M\right)  $ with the group
$\Gamma$ of deck transformations on $\widetilde{M}$ and therefore
$M=\widetilde{M}/\Gamma$. There is a natural $\Gamma$-action on $\partial
^{\ast}\widetilde{M}$: for $\xi\in\partial^{\ast}\widetilde{M}$ and $\gamma
\in\Gamma$%
\[
\left(  \gamma\cdot\xi\right)  \left(  x\right)  =\frac{\xi\left(  \gamma
^{-1}x\right)  }{\xi\left(  \gamma^{-1}o\right)  }.
\]
As a result for each $\gamma\in\Gamma$ we have the pushforward measure
$\gamma_{\ast}\nu$ such that $\gamma_{\ast}\nu\left(  E\right)  =\nu\left(
\gamma^{-1}\cdot E\right)  $ for any Borel set $E\subset\partial^{\ast
}\widetilde{M}$. By the definition of $\nu$ and a change of variables
\begin{align*}
1  & =\int_{\partial^{\ast}\widetilde{M}}\xi\left(  \gamma^{-1}x\right)
d\nu\left(  \xi\right) \\
& =\int_{\partial^{\ast}\widetilde{M}}\left(  \gamma\cdot\xi\right)  \left(
x\right)  \xi\left(  \gamma^{-1}o\right)  d\nu\left(  \xi\right)  .\\
& =\int_{\partial^{\ast}\widetilde{M}}\eta\left(  x\right)  \frac{1}%
{\eta\left(  \gamma o\right)  }d\gamma_{\ast}\nu\left(  \eta\right)  .
\end{align*}
By the uniqueness of $\nu$ we have
\[
d\gamma_{\ast}\nu\left(  \xi\right)  =\xi\left(  \gamma o\right)  d\nu\left(
\xi\right)  .
\]
We define%

\[
\phi\left(  x,y\right)  =-\int_{\partial^{\ast}\widetilde{M}}\xi\left(
y\right)  \log\frac{\xi\left(  x\right)  }{\xi\left(  y\right)  }d\nu\left(
\xi\right)  .
\]
It is easy to show that $\phi\left(  \gamma x,\gamma y\right)  =\phi\left(
x,y\right)  $ for any $\gamma\in\Gamma$ and%
\begin{align*}
\phi\left(  x,x\right)   & =0,\\
\Delta_{y}\phi\left(  x,y\right)   & =\int_{\partial^{\ast}\widetilde{M}}%
\xi\left(  y\right)  \left\vert \nabla\log\xi\left(  y\right)  \right\vert
^{2}d\nu\left(  \xi\right)  .
\end{align*}
By Yau's gradient estimate we have
\begin{align*}
\left\vert \phi\left(  x,y\right)  \right\vert  & \leq Cd\left(  x,y\right)
,\\
\left\vert \nabla_{y}\phi\left(  x,y\right)  \right\vert  & \leq Cd\left(
x,y\right)  ,\\
\left\vert \Delta_{y}\phi\left(  x,y\right)  \right\vert  & \leq C.
\end{align*}
Let $p\left(  t,x,y\right)  $ be the heat kernel on $\widetilde{M}$. For any
$\gamma\in\Gamma$%
\[
p\left(  t,\gamma x,\gamma y\right)  =p\left(  t,x,y\right)  .
\]
We define
\[
u\left(  \tau,x\right)  =\frac{1}{\tau}\int_{\widetilde{M}}\phi\left(
x,y\right)  p\left(  \tau,x,y\right)  dv\left(  y\right)  .
\]
It is easy to see that $u$ descends to $M$. Indeed,%
\begin{align*}
u\left(  \tau,\gamma x\right)   & =\frac{1}{\tau}\int_{\widetilde{M}}%
\phi\left(  \gamma x,y\right)  p\left(  \tau,\gamma x,y\right)  dv\left(
y\right) \\
& =\frac{1}{\tau}\int_{\widetilde{M}}\phi\left(  x,\gamma^{-1}y\right)
p\left(  \tau,\gamma x,y\right)  dv\left(  y\right) \\
& =\frac{1}{\tau}\int_{\widetilde{M}}\phi\left(  x,z\right)  p\left(
\tau,\gamma x,\gamma z\right)  dv\left(  z\right) \\
& =\frac{1}{\tau}\int_{\widetilde{M}}\phi\left(  x,z\right)  p\left(
\tau,x,z\right)  dv\left(  z\right) \\
& =u\left(  \tau,\gamma x\right)  .
\end{align*}
We can rewrite%
\begin{align*}
u\left(  \tau,x\right)   & =\frac{1}{\tau}\int_{\widetilde{M}}\phi\left(
x,y\right)  p\left(  \tau,x,y\right)  dv\left(  y\right) \\
& =\frac{1}{\tau}\int_{\widetilde{M}}\phi\left(  x,y\right)  \left(  \int
_{0}^{\tau}\frac{\partial}{\partial t}p\left(  t,x,y\right)  dt\right)
dv\left(  y\right) \\
& =\frac{1}{\tau}\int_{0}^{\tau}\left(  \int_{\widetilde{M}}\phi\left(
x,y\right)  \Delta_{y}p\left(  t,x,y\right)  dv\left(  y\right)  \right)  dt\\
& =\frac{1}{\tau}\int_{0}^{\tau}\left(  \int_{\widetilde{M}}\Delta_{y}%
\phi\left(  x,y\right)  p\left(  t,x,y\right)  dv\left(  y\right)  \right)
dt\\
& =\int_{\partial^{\ast}\widetilde{M}}\left[  \frac{1}{\tau}\int_{0}^{\tau
}\left(  \int_{\widetilde{M}}\xi\left(  y\right)  \left\vert \nabla\log
\xi\left(  y\right)  \right\vert ^{2}p\left(  t,x,y\right)  dv\left(
y\right)  \right)  dt\right]  d\nu\left(  \xi\right) \\
& =\int_{\widetilde{M}\times\partial^{\ast}\widetilde{M}\times\left[
0,1\right]  }\xi\left(  y\right)  \left\vert \nabla\log\xi\left(  y\right)
\right\vert ^{2}p\left(  \tau s,x,y\right)  dv\left(  y\right)  d\nu\left(
\xi\right)  ds
\end{align*}
i.e.%
\begin{equation}
u\left(  \tau,x\right)  =\int_{\widetilde{M}\times\partial^{\ast}\widetilde
{M}\times\left[  0,1\right]  }\xi\left(  y\right)  \left\vert \nabla\log
\xi\left(  y\right)  \right\vert ^{2}p\left(  \tau s,x,y\right)  dv\left(
y\right)  d\nu\left(  \xi\right)  ds.\label{eu}%
\end{equation}
Moreover%
\begin{align*}
\frac{\partial u}{\partial\tau}\left(  \tau,x\right)   & =\int_{\widetilde
{M}\times\partial^{\ast}\widetilde{M}\times\left[  0,1\right]  }\xi\left(
y\right)  \left\vert \nabla\log\xi\left(  y\right)  \right\vert ^{2}%
s\frac{\partial p}{\partial t}\left(  \tau s,x,y\right)  dv\left(  y\right)
d\nu\left(  \xi\right)  ds\\
& =\int_{\widetilde{M}\times\partial^{\ast}\widetilde{M}\times\left[
0,1\right]  }\xi\left(  y\right)  \left\vert \nabla\log\xi\left(  y\right)
\right\vert ^{2}s\Delta_{x}p\left(  \tau s,x,y\right)  dv\left(  y\right)
d\nu\left(  \xi\right)  ds\\
& =\Delta\left(  \int_{\widetilde{M}\times\partial^{\ast}\widetilde{M}%
\times\left[  0,1\right]  }\xi\left(  y\right)  \left\vert \nabla\log
\xi\left(  y\right)  \right\vert ^{2}sp\left(  \tau s,x,y\right)  dv\left(
y\right)  d\nu\left(  \xi\right)  ds\right)  .
\end{align*}
As a result $\int_{M}u\left(  \tau,x\right)  dx$ is independent of $\tau$. Let
$dm$ be the normalized volume form.

\begin{definition}
The number%
\[
\beta\left(  \widetilde{M}\right)  =\int_{M}u\left(  \tau,x\right)  dm\left(
x\right)
\]
is called the Kaimanovich entropy.
\end{definition}

\bigskip

\bigskip By (\ref{eu}) it is clear that
\[
u\left(  \tau,x\right)  \rightarrow\int_{\partial^{\ast}\widetilde{M}}%
\xi\left(  x\right)  \left\vert \nabla\log\xi\left(  x\right)  \right\vert
^{2}d\nu\left(  \xi\right)
\]
as $\tau\rightarrow0$. Hence we also obtain the following formula for $\beta$
from \cite{K}.

\begin{proposition}
\label{beta2}%
\[
\beta\left(  \widetilde{M}\right)  =\int_{M}\left(  \int_{\partial^{\ast
}\widetilde{M}}\xi\left(  x\right)  \left\vert \nabla\log\xi\left(  x\right)
\right\vert ^{2}d\nu\left(  \xi\right)  \right)  dm\left(  x\right)  .
\]

\end{proposition}

Our main result of this section is the following sharp estimate for the
entropy under a lower Ricci bound.

\begin{theorem}
\label{beta}Let $\left(  M^{n},g\right)  $ be a compact Riemannian manifold
with $Ric\geq-\left(  n-1\right)  $ and $\pi:\widetilde{M}\rightarrow M$ its
universal covering. Then $\beta\left(  \widetilde{M}\right)  \leq\left(
n-1\right)  ^{2}$ and equality holds iff $\widetilde{M}$ is isometric to the
hyperbolic space $\mathbb{H}^{n}$.
\end{theorem}

We first prove the following theorem on manifolds with $Ric\geq0$ which may be
of independent interest.

\begin{theorem}
\label{Rn}Let $\left(  \Sigma^{n-1},g\right)  $ be a simply connected complete
Riemannian manifold with $Ric\geq0,n\geq3$. Suppose there is a smooth
positive, nonconstant function $u:\Sigma\rightarrow%
\mathbb{R}
$ such that%
\begin{equation}
\left\{
\begin{array}
[c]{c}%
D^{2}u=u\left(  1+\lambda\right)  g,\\
\frac{\left\vert \nabla u\right\vert ^{2}}{u^{2}}=1-\lambda^{2}%
\end{array}
\right. \label{dit}%
\end{equation}

for some smooth function $\lambda$, then $\left(  \Sigma^{n-1},g\right)  $ is
isometric to $%
\mathbb{R}
^{n-1}$ and on $%
\mathbb{R}
^{n-1}\ $
\[
u=c\left(  1+\left\vert x-x_{0}\right\vert ^{2}\right)
\]
for some constant $c>0$ and $x_{0}\in%
\mathbb{R}
^{n-1}$.

\begin{proof}
We divide the proof into several steps. Set $\mu=u\left(  1+\lambda\right)  $.

\textbf{Step 1: }Since $u$ is nonconstant and satisfies an elliptic equation,
the set $\{u\neq0\}$ is open and dense. This is also true of the set
$\{\lambda\neq0\}$. Indeed, if $\lambda=0$ on some open set $U$, then
$\left\vert \nabla u\right\vert ^{2}=u^{2},D^{2}u=ug$ on $U$. Using the
Bochner formula
\[
\frac{1}{2}\Delta\left\vert \nabla u\right\vert ^{2}=\left\vert D^{2}%
u\right\vert ^{2}+\left\langle \nabla u,\nabla\Delta u\right\rangle
+Ric\left(  \nabla u,\nabla u\right)
\]
we then easily get $\left(  n-2\right)  \left\vert \nabla u\right\vert
^{2}+Ric\left(  \nabla u,\nabla u\right)  =0$ by a simple computation. Since
$Ric\geq0$, we conclude that $u$ is constant on $U$. A contradiction.

\textbf{Step 2}: For any vector filed $X$ we have%
\begin{equation}
\frac{1}{2}X\left\vert \nabla u\right\vert ^{2}=\left\langle \nabla_{X}\nabla
u,\nabla u\right\rangle =\mu\left\langle X,\nabla u\right\rangle .\label{diff}%
\end{equation}

Taking $X$ to be $\nabla u$ and using the second equation of (\ref{dit})
yields%
\begin{align*}
\mu\left\vert \nabla u\right\vert ^{2}  & =\frac{1}{2}\left\langle \nabla
u,\nabla\left\vert \nabla u\right\vert ^{2}\right\rangle \\
& =\frac{1}{2}\left\langle \nabla u,\nabla\left(  u^{2}\left(  1-\lambda
^{2}\right)  \right)  \right\rangle \\
& =\left(  1-\lambda^{2}\right)  u\left\vert \nabla u\right\vert ^{2}-\lambda
u^{2}\left\langle \nabla u,\nabla\lambda\right\rangle .
\end{align*}
Hence%
\[
\lambda\left\langle \nabla u,\nabla\mu\right\rangle =0.
\]

On the other hand again from (\ref{diff}) for any $X$ with $\left\langle
X,\nabla u\right\rangle =0$ we have $\left\langle X,\nabla\left\vert \nabla
u\right\vert ^{2}\right\rangle =0$ and hence $\left\langle X,\nabla
\mu\right\rangle =0$ as $\nabla\mu$ is a linear combination of $\nabla u$ and
$\nabla\left\vert \nabla u\right\vert ^{2}$. Therefore $\nabla\mu=0$ on the
set $\{\nabla u\neq0,\lambda\neq0\}$. Since this set is open and dense in
$\Sigma\,$, we conclude that $\mu$ is a positive constant.

\textbf{Step 3}: Since $\nabla_{X}\nabla u=\mu X$ and $\mu$ is constant, it is
easy to see
\begin{equation}
R\left(  X,Y,Z,\nabla u\right)  =0.\label{curv}%
\end{equation}

If $n=3$, then $\Sigma$ is flat and hence isometric to $%
\mathbb{R}
^{2}$. In the remaining steps we assume $n>3$.

\textbf{Step 4}: We show that each regular level set of $u$ is compact. Let
$S=u^{-1}\left(  c\right)  $ with $c$ a regular value. The unit normal of $S$
is $\nu=\nabla u/\left\vert \nabla u\right\vert $ and its second fundamental
form is given by
\begin{align*}
\Pi\left(  X,Y\right)   & =\left\langle \nabla_{X}\nu,Y\right\rangle \\
& =\frac{D^{2}u\left(  X,Y\right)  }{\left\vert \nabla u\right\vert }\\
& =\frac{\mu}{\left\vert \nabla u\right\vert }\left\langle X,Y\right\rangle
\end{align*}
for $X,Y$ tangent to $S$. Similarly
\[
X\left\vert \nabla u\right\vert =\frac{D^{2}u\left(  X,\nabla u\right)
}{\left\vert \nabla u\right\vert }=\frac{\mu Xu}{\left\vert \nabla
u\right\vert }=0
\]
for $X$ tangent to $S$. As a result $\left\vert \nabla u\right\vert $ and
$a=\frac{\mu}{\left\vert \nabla u\right\vert }$ are positive constants along
$S $. We compute the intrinsic Ricci curvature of $S$%
\begin{align*}
Ric_{S}\left(  X,X\right)   & =Ric\left(  X,X\right)  -R\left(  X,\nu
,X,\nu\right)  +\left(  n-3\right)  a^{2}\left\vert X\right\vert ^{2}\\
& \geq\left(  n-3\right)  a^{2}\left\vert X\right\vert ^{2}.
\end{align*}
It follows that $S$ is compact by Bonnet-Myers theorem. Since $\Sigma$ is
simply connected, each connected component of $S$ separates $\Sigma$ into two components.

\textbf{Step 5}: The first equation of (\ref{dit}) implies $D^{2}u>0$, i.e.
$u$ is convex. Then it is easy to see that $S$ is connected and $\left\{
u\leq c\right\}  $ is the inner component of $\Sigma-S$ and hence compact. In
other words $u$ is proper. Let $p$ be a point where $u$ achieves its minimum
$\mu/2$. For $X\in\mathcal{S}_{p}M$ let $\gamma\left(  t\right)  $ be the
geodesic with $\overset{\cdot}{\gamma}\left(  0\right)  =X$. Then $f\left(
t\right)  =u\circ\gamma\left(  t\right)  $ satisfies
\[
f^{\prime\prime}\left(  t\right)  =\mu,f\left(  0\right)  =\mu/2,f^{\prime
}\left(  0\right)  =0.
\]
Hence $f\left(  t\right)  =\mu\left(  1+t^{2}\right)  $. In other words%
\[
u\left(  \exp_{p}rX\right)  =\frac{\mu}{2}\left(  1+r^{2}\right)  .
\]
In geodesic polar coordinates $r$ is the distance function to $p$. The first
equation of (\ref{dit}) then simply means
\[
D^{2}r^{2}=2g
\]
at least within the cut locus. It is then easy to show that $\left(
\Sigma^{n-1},g\right)  $ is flat and hence isometric to $%
\mathbb{R}
^{n-1}$.
\end{proof}
\end{theorem}

\begin{proof}
[Proof of Theorem \ref{beta}]By Lemma \ref{ge} we have for any $\xi\in
\partial^{\ast}$ $\widetilde{M}$%
\[
\left\vert \nabla\log\xi\left(  x\right)  \right\vert \leq n-1.
\]
Hence, using Proposition \ref{beta2}%
\begin{align*}
\beta\left(  \widetilde{M}\right)   & =\int_{M}\left(  \int_{\partial^{\ast
}\widetilde{M}}\xi\left(  x\right)  \left\vert \nabla\log\xi\left(  x\right)
\right\vert ^{2}d\nu\left(  \xi\right)  \right)  dm\left(  x\right) \\
& \leq\left(  n-1\right)  ^{2}\int_{M}\left(  \int_{\partial^{\ast}%
\widetilde{M}}\xi\left(  x\right)  d\nu\left(  \xi\right)  \right)  dm\left(
x\right) \\
& =\left(  n-1\right)  ^{2}.
\end{align*}
If $\beta\left(  \widetilde{M}\right)  =\left(  n-1\right)  ^{2}$, then there
exists $A\subset\partial^{\ast}$ $\widetilde{M}$ with $\nu\left(
\partial^{\ast}\widetilde{M}\backslash A\right)  =0$ such that for any $\xi\in
A$
\begin{equation}
\left\vert \nabla\log\xi\left(  x\right)  \right\vert \equiv n-1.\label{ae}%
\end{equation}

Let $\xi\in A$ be such a point. By Lemma we have $\widetilde{M}=%
\mathbb{R}
\times\Sigma^{n-1}$ with $g=dt^{2}+e^{2t}g_{\Sigma},\xi=\exp\left[  -\left(
n-1\right)  t\right]  $, where $\left(  \Sigma,g_{\Sigma}\right)  $ is a
complete Riemannian manifold with $Ric\geq0$. Notice that $o\in\{0\}\times
\Sigma$. Moreover $\Sigma$ is simply connected as $\widetilde{M}$ is. If
$n=2$, we are done. From then on we assume $n\geq3$. It is clear that
$A\backslash\{\xi\}$ is not empty by (\ref{defnu}). Let $\eta\in
A\backslash\{\xi\}$. We know that $\phi=\log\eta$ satisfies%
\begin{align*}
\left\vert \nabla\phi\right\vert  & =\left(  n-1\right)  ,\Delta\phi=-\left(
n-1\right)  ^{2},\\
D^{2}\phi & =-\left(  n-1\right)  \left[  g-\frac{1}{\left(  n-1\right)  ^{2}%
}d\phi\otimes d\phi\right]  .
\end{align*}
Let $\psi=\exp\left(  -\frac{\phi}{n-1}\right)  $. Then a simple computation
shows%
\begin{align*}
\frac{\left\vert \nabla\psi\right\vert ^{2}}{\psi^{2}}  & =1,\\
D^{2}\psi & =\psi g.
\end{align*}
Notice that $\Sigma$, view as the $t=0$ slice in $\widetilde{M}=%
\mathbb{R}
\times\Sigma^{n-1}$ is umbilic in the sense that the second fundamental form
w.r.t. the unit normal $\frac{\partial}{\partial t}$ equals the metric $h$. As
a result $u=\psi|_{\Sigma}$, the restriction of $\psi$ on $\Sigma$ satisfies
the following equations%
\begin{align*}
\frac{\left\vert \nabla u\right\vert ^{2}}{u^{2}}  & =1-\lambda^{2},\\
D^{2}u  & =u\left(  1+\lambda\right)  g_{0},
\end{align*}
where $\lambda=\frac{\partial\log\psi}{\partial t}$ along $\Sigma$. We claim
that $u$ is not constant on $\Sigma$. If this is NOT true, then $\psi\left(
0,x\right)  \equiv1$ as $o\in\{0\}\times\Sigma$. Then either $\nabla
\psi\left(  0,x\right)  =\frac{\partial}{\partial t}$ or $\nabla\psi\left(
0,x\right)  =-\frac{\partial}{\partial t}$ along $\Sigma$. In the first case
$\psi\left(  t,x\right)  $ satisfies%
\[
\frac{\partial^{2}\psi}{\partial t^{2}}\left(  t,x\right)  =\psi\left(
t,x\right)  ,\psi\left(  0,x\right)  \equiv1,\frac{\partial\psi}{\partial
t}\left(  0,x\right)  =1
\]
and hence $\psi\left(  t,x\right)  =e^{t}$. This then implies that $\eta
=\exp\left[  -\left(  n-1\right)  t\right]  =\xi$, a contradiction. In the
second case we get $\psi\left(  t,x\right)  =e^{-t}$ and $\eta=\exp\left[
\left(  n-1\right)  t\right]  =\frac{1}{\xi}$. But this is not harmonic as it
is easy to check: $\Delta\eta=2\frac{\left\vert \nabla\xi\right\vert ^{2}}%
{\xi^{3}}=\frac{2\left(  n-1\right)  ^{2}}{\xi}$.

Since $u$ is not constant on $\Sigma$, by Theorem \ref{Rn} $\left(
\Sigma,g_{\Sigma}\right)  $ is isometric to $%
\mathbb{R}
^{n-1}$. Therefore $\widetilde{M}$ is isometric to $\mathbb{H}^{n}$.
\end{proof}

\begin{remark}
For the hyperbolic space $\mathbb{H}^{n}$ the Martin boundary is the same as
the ideal boundary. In the ball model $\mathbb{H}^{n}$ is simply the unit ball
$B^{n}$ in $%
\mathbb{R}
^{n}$ with the metric $g=\frac{4}{\left(  1-\left\vert x\right\vert
^{2}\right)  }dx^{2}$. We take the base point to be the origin. The ideal
boundary is the unit sphere $S^{n-1}$. For any $\xi\in S^{n-1}$ there
corresponds to a normalized minimal positive function%
\[
h_{\xi}\left(  x\right)  =\left(  \frac{1-\left\vert x\right\vert ^{2}%
}{\left\vert x-\xi\right\vert ^{2}}\right)  ^{n-1}.
\]
It is easy to verify that they all satisfy $\left\vert \nabla\log h_{\xi
}\left(  x\right)  \right\vert \equiv\left(  n-1\right)  $.
\end{remark}

\section{Proof of the main theorem}

\bigskip First we need another remarkable formula from \cite{K}.

\begin{theorem}%
\[
\beta\left(  \widetilde{M}\right)  =-\lim_{t\rightarrow\infty}\frac{1}{t}%
\int_{\widetilde{M}}p\left(  t,x,y\right)  \log p\left(  t,x,y\right)
dv\left(  y\right)  .
\]

\end{theorem}

We also need the following lemma from \cite{L1}.

\begin{lemma}
\label{Le}Let $\left(  M,g\right)  $ be a compact Riemannian manifold and
$\pi:\widetilde{M}\rightarrow M$ its universal covering. Then $\beta\left(
\widetilde{M}\right)  \geq4\lambda_{0}\left(  \widetilde{M}\right)  $.
\end{lemma}

\begin{proof}
Since the proof is short and instructive, we present it for the convenience of
the reader. For $\varepsilon>0$%
\begin{align*}
& -\frac{1}{t}\int_{\widetilde{M}}p\left(  t,x,y\right)  \log p\left(
t,x,y\right)  dv\left(  y\right)  +\frac{1}{t}\int_{\widetilde{M}}p\left(
\varepsilon,x,y\right)  \log p\left(  \varepsilon,x,y\right)  dv\left(
y\right) \\
& =-\frac{1}{t}\int_{\widetilde{M}}\int_{\varepsilon}^{t}\left(  \log p\left(
s,x,y\right)  +1\right)  \frac{\partial p}{\partial s}\left(  s,x,y\right)
dsdv\left(  y\right) \\
& =-\frac{1}{t}\int_{\widetilde{M}}\int_{\varepsilon}^{t}\left(  \log p\left(
s,x,y\right)  +1\right)  \Delta_{y}p\left(  s,x,y\right)  dsdv\left(  y\right)
\\
& =-\frac{1}{t}\int_{\widetilde{M}}\int_{\varepsilon}^{t}\Delta_{y}\left(
\log p\left(  s,x,y\right)  +1\right)  p\left(  s,x,y\right)  dsdv\left(
y\right) \\
& =\frac{1}{t}\int_{\varepsilon}^{t}\int_{\widetilde{M}}\frac{\left\vert
\nabla_{y}p\left(  s,x,y\right)  \right\vert ^{2}}{p\left(  s,x,y\right)
}dsdv\left(  y\right)  ds\\
& =\frac{4}{t}\int_{\varepsilon}^{t}\int_{\widetilde{M}}\left\vert \nabla
_{y}\sqrt{p\left(  s,x,y\right)  }\right\vert dv\left(  y\right)  ds\\
& \geq\frac{4}{t}\lambda_{0}\left(  \widetilde{M}\right)  \int_{\varepsilon
}^{t}\int_{\widetilde{M}}p\left(  s,x,y\right)  dv\left(  y\right)  ds\\
& =\frac{4}{t}\lambda_{0}\left(  \widetilde{M}\right)  \left(  t-\varepsilon
\right)  .
\end{align*}
In the last step we use the fact that $\widetilde{M}$ is stochastically
complete, i.e.%
\[
\int_{\widetilde{M}}p\left(  s,x,y\right)  dv\left(  y\right)  =1.
\]
Letting $t\rightarrow\infty$ yields $\beta\left(  \widetilde{M}\right)
\geq4\lambda_{0}\left(  \widetilde{M}\right)  $.
\end{proof}

\bigskip We now prove our main theorem

\begin{theorem}
Let $\left(  M^{n},g\right)  $ be a compact Riemannian manifold with
$Ric\geq-\left(  n-1\right)  $ and $\pi:\widetilde{M}\rightarrow M$ its
universal covering. Then

\begin{enumerate}
\item $\lambda_{0}\left(  \widetilde{M}\right)  \leq\left(  n-1\right)
^{2}/4$,

\item If equality holds, then $\widetilde{M}$ is isometric to the hyperbolic
space $\mathbb{H}^{n}$.
\end{enumerate}
\end{theorem}

\begin{proof}
The first part is well known and follows from a theorem of Cheng \cite{C} or
the inequality $\lambda_{0}\left(  \widetilde{M}\right)  \leq h^{2}/4$, where
$h$ is the volume entropy. The second part clearly follows from Lemma \ref{Le}
and Theorem \ref{beta}.
\end{proof}

\end{document}